\documentclass[12pt,reqno]{amsart}
\usepackage{amssymb,amsmath,amsthm,hyperref}%took out hyperref, multicol
\newtheorem{theorem}{Theorem}
\newtheorem{lemma}[theorem]{Lemma}
\newtheorem{corollary}[theorem]{Corollary}
\newtheorem{proposition}[theorem]{Proposition}

\theoremstyle{remark}

\theoremstyle{definition}

\newtheorem{definition}[theorem]{Definition}

\numberwithin{theorem}{section} \numberwithin{equation}{section}
\numberwithin{example}{section}

\title{On the tenth-order mock theta functions}

\author{Eric T. Mortenson}

\begin{document}

\date{9 October 2016}



\begin{abstract}
Using properties of Appell--Lerch functions, we give insightful proofs for six of Ramanujan's identities for the tenth-order mock theta functions.  
\end{abstract}

\address{Max-Planck-Institut f\"ur Mathematik, Vivatsgasse 7, 53111 Bonn, Germany}
\email{etmortenson@gmail.com}
\maketitle
\setcounter{section}{-1}
\setcounter{tocdepth}{1} % omits subsections from TOC

\section{Notation}

 Let $q:=q_{\tau}=e^{2 \pi i \tau}$, $\tau\in\mathbb{H}:=\{ z\in \mathbb{C}| \textup{Im}(z)>0 \}$, and define $\mathbb{C}^*:=\mathbb{C}-\{0\}$.  Recall
\begin{gather}
(x)_n=(x;q)_n:=\prod_{i=0}^{n-1}(1-q^ix), \ \ (x)_{\infty}=(x;q)_{\infty}:=\prod_{i\ge 0}(1-q^ix),\notag \\
 j(x;q):=(x)_{\infty}(q/x)_{\infty}(q)_{\infty}=\sum_{n=-\infty}^{\infty}(-1)^nq^{\binom{n}{2}}x^n,\notag\\
{\text{and }} j(x_1,x_2,\dots,x_n;q):=j(x_1;q)j(x_2;q)\cdots j(x_n;q),\notag
\end{gather}
where in the penultimate line the equivalence of product and sum follows from Jacobi's triple product identity.    Here $a$ and $m$ are integers with $m$ positive.  Define
\begin{gather*}
J_{a,m}:=j(q^a;q^m), \ \ J_m:=J_{m,3m}=\prod_{i\ge 1}(1-q^{mi}), \ {\text{and }}\overline{J}_{a,m}:=j(-q^a;q^m).
\end{gather*}
We will use the following definition of an Appell--Lerch function \cite{HM, Zw2}
\begin{equation}
m(x,q,z):=\frac{1}{j(z;q)}\sum_{r=-\infty}^{\infty}\frac{(-1)^rq^{\binom{r}{2}}z^r}{1-q^{r-1}xz}.\label{equation:mdef-eq}
\end{equation}

\section{Introduction}
Ramanujan's mock theta functions have puzzled and fascinated mathematicians for decades.  After work of Zwegers \cite{Zw2}, the functions may be viewed as holomorphic parts of weak Maass forms \cite{BrO2, BrOR}.   Here we will revisit the tenth-order mock theta functions
{\allowdisplaybreaks \begin{align}
{\phi}(q)&=\sum_{n\ge 0}\frac{q^{\binom{n+1}{2}}}{(q;q^2)_{n+1}}, \ \ {\psi}(q)=\sum_{n\ge 0}\frac{q^{\binom{n+2}{2}}}{(q;q^2)_{n+1}}, \label{equation:tenth-fns}\\ 
& \ \ \ \ \ {X}(q)=\sum_{n\ge 0}\frac{(-1)^nq^{n^2}}{(-q;q)_{2n}}, \ \  {\chi}(q)=\sum_{n\ge 0}\frac{(-1)^nq^{(n+1)^2}}{(-q;q)_{2n+1}},\notag
\end{align}}%
which satisfy many identities such as the slightly-rewritten \cite{C1,C2}
{\allowdisplaybreaks \begin{align}
q^{2}\phi(q^9)-\frac{\psi(\omega q)-\psi(\omega^2 q)}{\omega - \omega^2}
&=-q\frac{J_{1,2}}{J_{3,6}}\frac{J_{3,15}J_{6}}{J_{3}},\label{equation:tenth-id-1}\\
q^{-2}\psi(q^9)+\frac{\omega \phi(\omega q)-\omega^2\phi(\omega^2 q)}{\omega - \omega^2}
&=\frac{J_{1,2}}{J_{3,6}}\frac{J_{6,15}J_{6}}{J_{3}},\label{equation:tenth-id-2}\\
X(q^9)-\frac{\omega \chi(\omega q)-\omega^2\chi(\omega^2 q)}{\omega - \omega^2}
&=\frac{\overline{J}_{1,4}}{\overline{J}_{3,12}}\frac{J_{18,30}J_{3}}{J_{6}},\label{equation:tenth-id-3}\\
\chi(q^9)+q^{2}\frac{ X(\omega q)-X(\omega^2 q)}{\omega - \omega^2}&=-q^3\frac{\overline{J}_{1,4}}{\overline{J}_{3,12}}
\frac{J_{6,30}J_{3}}{J_{6}},\label{equation:tenth-id-4}
\end{align}}%
where $\omega$ is a primitive third root of unity, as well as the \cite{C3}
\begin{align}
\phi(q)-q^{-1}\psi(-q^4)+q^{-2}\chi(q^8)&=\frac{\overline{J}_{1,2}j(-q^2;-q^{10})}{J_{2,8}},\label{equation:RLN-id-five}\\
\psi(q)+q\phi(-q^4)+X(q^8)&=\frac{\overline{J}_{1,2}j(-q^6;-q^{10})}{J_{2,8}}.\label{equation:RLN-id-six}
\end{align}

The six identities were originally found in the lost notebook \cite{RLN} but first proved by Choi \cite{C1, C2, C3}.  Identities (\ref{equation:tenth-id-1})--(\ref{equation:tenth-id-4}) were recently given significantly shorter proofs by Zwegers \cite{Zw3}.  In this note, we will give short proofs of Ramanujan's six identities for the tenth-order mock theta functions using a recent result of Hickerson and the author: 
\begin{theorem} \label{theorem:msplit-general-n} \cite[Theorem $3.5$]{HM} For generic $x,z,z'\in \mathbb{C}^*$ 
{\allowdisplaybreaks \begin{align}
 D_n(x,q,z,z')=z' J_n^3  \sum_{r=0}^{n-1}
\frac{q^{{\binom{r}{2}}} (-xz)^r
j\big(-q^{{\binom{n}{2}+r}} (-x)^n z z';q^n\big)
j(q^{nr} z^n/z';q^{n^2})}
{j(xz;q) j(z';q^{n^2}) j\big(-q^{{\binom{n}{2}}} (-x)^n z';q^n)j(q^r z;q^n\big )},
\end{align}}%
where
\begin{equation}
D_n(x,q,z,z'):=m(x,q,z) - \sum_{r=0}^{n-1} q^{{-\binom{r+1}{2}}} (-x)^r m\big({-}q^{{\binom{n}{2}-nr}} (-x)^n, q^{n^2}, z' \big).
\label{equation:Dn-def}
\end{equation}
\end{theorem}
\noindent In so doing, we will keep this note as independent as possible from Choi's work.  Although we will take Choi's Hecke-type double-sum expansions of the four functions $\phi$, $\psi$, $X$, and $\chi$, that is where the similarity of our papers and any dependence ends.

In Section \ref{section:notation}, we recall background information.  In Section \ref{section:hecke}, we take Choi's Hecke-type double-sum expansions of the four functions and use a specialization of \cite[Theorem $1.3$]{HM} to express the double-sums in terms of the $m(x,q,z)$ function.  We see in Section \ref{section:Dn-forms} that once identities (\ref{equation:tenth-id-1})--(\ref{equation:RLN-id-six}) have been written in terms of Appell--Lerch functions, that the identities may be written in terms of specializations of  the $D_n(x,q,z,z')$ function, so perhaps Ramanujan knew something along the lines of \cite[Theorem $3.5$]{HM}.  In Section \ref{section:single-quotients}, we evaluate the specializations of (\ref{equation:Dn-def}) in terms of single-quotient theta functions.  In Section \ref{section:five-six}, we prove identities (\ref{equation:RLN-id-five}) and (\ref{equation:RLN-id-six}).  In Section \ref{section:first-and-second}, we prove (\ref{equation:tenth-id-1}) and (\ref{equation:tenth-id-2}), and in Section \ref{section:third-and-fourth}, we prove (\ref{equation:tenth-id-3}) and (\ref{equation:tenth-id-4}).

For the interested reader, we point out that \cite[Theorem $3.5$]{HM} and its parent identity \cite[Theorem $3.9$]{HM} also give an elegant proof \cite{HM2} of celebrated results of Bringmann {\em et al}. on Dyson's ranks and Maass forms \cite{BrO2, BrOR}.

\section{Preliminaries}\label{section:notation}

We have the general identities:
\begin{subequations}
{\allowdisplaybreaks \begin{gather}
j(q^n x;q)=(-1)^nq^{-\binom{n}{2}}x^{-n}j(x;q), \ \ n\in\mathbb{Z},\label{equation:j-elliptic}\\
j(x;q)=j(q/x;q)=-xj(x^{-1};q)\label{equation:j-inv},\\
j(x;q)={J_1}j(x,qx,\dots,q^{n-1}x;q^n)/{J_n^n} \ \ {\text{if $n\ge 1$,}}\label{equation:j-mod-inc}\\
j(x;-q)={j(x;q^2)j(-qx;q^2)}/{J_{1,4}},\label{equation:j-neg-mod}\\
j(z;q)=\sum_{k=0}^{m-1}(-1)^kq^{\binom{k}{2}}z^kj((-1)^{m+1}q^{\binom{m}{2}+mk}z^m;q^{m^2}),\label{equation:j-split}\\
j(x^n;q^n)={J_n}j(x,\zeta_nx,\dots,\zeta_n^{n-1}x;q^n)/{J_1^n}\ \ {\text{if $n\ge 1$}}.\label{equation:j-mod-dec}
\end{gather}}%
\end{subequations}
where $\zeta_n$ is a primitive $n$-th root of unity.   We state additional useful results:
\begin{proposition} \cite[Theorems $1.0$, $1.1$, and $1.2$]{H1}  For generic $x,y,z\in \mathbb{C}^*$ 
 \begin{subequations}
{\allowdisplaybreaks \begin{gather}
j(qx^3;q^3)+xj(q^2x^3;q^3)=j(-x;q)j(qx^2;q^2)/J_2={J_1j(x^2;q)}/{j(x;q)},\label{equation:H1Thm1.0}\\
j(x;q)j(y;q)=j(-xy;q^2)j(-qx^{-1}y;q^2)-xj(-qxy;q^2)j(-x^{-1}y;q^2),\label{equation:H1Thm1.1}\\
j(-x;q)j(y;q)+j(x;q)j(-y;q)=2j(xy;q^2)j(qx^{-1}y;q^2).\label{equation:H1Thm1.2B}
\end{gather}}
\end{subequations}
\end{proposition}
We recall the three-term Weierstrass relation for theta functions \cite[(1.)]{We}, \cite{Ko}:
\begin{proposition}\label{proposition:Weierstrass-id} For generic $a,b,c,d\in \mathbb{C}^*$
\begin{equation}
j(ac,a/c,bd,b/d;q)=j(ad,a/d,bc,b/c;q)+b/c \cdot j(ab,a/b,cd,c/d;q).\label{equation:Weierstrass}
\end{equation}
\end{proposition}

The Appell--Lerch function $m(x,q,z)$ satisfies several functional equations and identities, which we collect in the form of a proposition \cite{HM, Zw2}:  

\begin{proposition}  For generic $x,z\in \mathbb{C}^*$
{\allowdisplaybreaks \begin{subequations}
\begin{gather}
m(x,q,z)=m(x,q,qz),\label{equation:mxqz-fnq-z}\\
m(x,q,z)=x^{-1}m(x^{-1},q,z^{-1}),\label{equation:mxqz-flip}\\
m(x,q,z)=m(x,q,x^{-1}z^{-1}),\label{equation:mxqz-fnq-newz}\\
m(x,q,z_1)-m(x,q,z_0)=\frac{z_0J_1^3j(z_1/z_0;q)j(xz_0z_1;q)}{j(z_0;q)j(z_1;q)j(xz_0;q)j(xz_1;q)}.\label{equation:changing-z}
\end{gather}
\end{subequations}}
\end{proposition}

We point out the $n=2$ and $n=3$ specializations of \cite[Theorem $3.5$]{HM}: 
\begin{corollary} \label{corollary:msplitn2zprime} For generic $x,z,z'\in \mathbb{C}^*$ 
{\allowdisplaybreaks \begin{align}
D_2&(x,q,z,z') \label{equation:msplit2}\\
&=\frac{z'J_2^3}{j(xz;q)j(z';q^4)}\Big [
\frac{j(-qx^2zz';q^2)j(z^2/z';q^{4})}{j(-qx^2z';q^2)j(z;q^2)}-xz \frac{j(-q^2x^2zz';q^2)j(q^2z^2/z';q^{4})}{j(-qx^2z';q^2)j(qz;q^2)}\Big ],\notag
\end{align}}%
where
\begin{equation}
D_2(x,q,z,z'):=m(x,q,z)-m(-qx^2,q^4,z' )+q^{-1}xm(-q^{-1}x^2,q^4,z').\label{equation:D2-def}
\end{equation}
\end{corollary}

\begin{corollary} \label{corollary:msplitn3zprime} For generic $x,z,z'\in \mathbb{C}^*$ 
\begin{align}
D_3(x,q,z,z')&=\frac{z'J_3^3}{j(xz;q)j(z';q^{9})j(x^3z';q^3)}\Big [ 
\frac{1}{z}\frac{j(x^3zz';q^3)j(z^3/z';q^{9})}{j(z;q^3)}\label{equation:msplit3} \\
&\ \ \ \ \ -\frac{x}{q}\frac{j(qx^3zz';q^3)j(q^{3}z^3/z';q^{9})}{j(qz;q^3)}
+\frac{x^2z}{q}\frac{j(q^2x^3zz';q^3)j(q^{6}z^3/z';q^{9})}{j(q^2z;q^3)}\Big ],\notag
\end{align}
where
\begin{align}
D_3(x,q,z,z')&:=m(x,q,z)-m\Big (q^{3}x^3,q^{9},z'\Big )\label{equation:D3-def}\\
&\ \ \ \ \  +q^{-1}xm\Big (x^3,q^{9},z'\Big )-q^{-3}x^2m\Big (q^{-3}x^3,q^{9},z'\Big ).\notag
\end{align}
\end{corollary}
We present a result similar to \cite[Theorem $1.3$]{AH} and prove two theta function identities.  
\begin{theorem}  We have
\begin{equation}
j(x;q)j(y;q^6)=\sum_{i=-2}^{2}(-1)^iq^{(i^2-i)/2}x^ij(-q^{3i+9}x^3y^{-1};q^{15})j(q^{2i+1}x^2y;q^{10}).\label{equation:DTH}
\end{equation}
\end{theorem}
\begin{proof}  We write
\begin{align*}
j(x;q)j(y;q^6)&=\sum_{r\in\mathbb{Z}}(-1)^rq^{r(r-1)/2}x^r\cdot \sum_{s\in\mathbb{Z}}(-1)^sq^{3s(s-1)}y^s\\
&=\sum_{r,s\in\mathbb{Z}}(-1)^{r+s}q^{(r^2-r+6s^2-6s)/2}x^ry^s.
\end{align*}
Break this into five pieces, depending on $(r-2s)$ mod $5$.  Let $r=2s+5u+i$ with $-2\le i \le 2$.  Then let $s=v-u$, so $r=3u+2v+i:$
{\allowdisplaybreaks \begin{align*}
j(&x;q)j(y;q^6)\\
&=\sum_{i=-2}^2\sum_{u,v\in\mathbb{Z}}(-1)^{2u+3v+i}q^{(15u^2+(6i+3)u)/2+5v^2+(2i-4)v+(i^2-i)/2}x^{3u+2v+i}y^{-u+v}\\
&=\sum_{i=-2}^2(-1)^iq^{(i^2-i)/2}x^i\sum_{u\in\mathbb{Z}}q^{(15u^2+(6i+3)u)/2}(x^3y^{-1})^u
\sum_{v\in\mathbb{Z}}(-1)^vq^{5v^2+(2i-4)v}(x^2y)^v\\
&=\sum_{i=-2}^2(-1)^iq^{(i^2-i)/2}x^ij(-q^{3i+9}x^3y^{-1};q^{15})j(q^{2i+1}x^2y;q^{10}).\qedhere
\end{align*}}%
\end{proof}
\begin{corollary} We have
\begin{align}
j(x;q)j(-x^3;q^6)&=J_{3,15}\Big [ q^3x^{-2}j(-q^{-3}x^5;q^{10})-xj(-q^3x^5;q^{10})\Big ]\label{equation:DCH}\\
&\ \ \ \ \ +J_{6,15}\Big [ j(-qx^5;q^{10})-qx^{-1}j(-q^{-1}x^5;q^{10})\Big ].\notag
\end{align}
\end{corollary}
\begin{proof}  Substitute $y=-x^3$ in (\ref{equation:DTH}):
\begin{equation*}
j(x;q)j(-x^3;q^6)=\sum_{i=-2}^2(-1)^iq^{(i^2-i)/2}x^iJ_{3i+9,15}j(-q^{2i+1}x^5;q^{10}).
\end{equation*} 
The $i=2$ term is zero, and the other terms can be combined in pairs to give the stated results, using $J_{3,15}=J_{12,15}$ and $J_{6,15}=J_{9,15}$.
\end{proof}
\begin{corollary} \label{corollary:cor-two-ids} The following two identities are true,
\begin{align}
J_{1,5}J_{12,30}-qJ_{2,5}J_{6,30}&=J_{1,2}\overline{J}_{3,12}=J_{1}J_{1,6},\label{equation:id-1}\\
J_{4,10}J_{6,15}+qJ_{2,10}J_{3,15}&=\overline{J}_{1,4}J_{3,6}=J_2\overline{J}_{1,3}.\label{equation:id-2}
\end{align}
\end{corollary}
\begin{proof} The second equality of each identity is just a product rearrangement.  To prove (\ref{equation:id-1}), we first substitute $x\rightarrow q$, $q\rightarrow q^2$ in (\ref{equation:DCH}):
\begin{equation*}
J_{1,2}\overline{J}_{3,12}=J_{6,30}\Big ( q^{4}\overline{J}_{-1,20}-q\overline{J}_{11,20}\Big )
+J_{12,30}\Big ( \overline{J}_{7,20}-q\overline{J}_{3,20}\Big ).
\end{equation*}
By (\ref{equation:j-split}) with $m=2$, we have
\begin{equation*}
J_{1,5}=\overline{J}_{7,20}-q\overline{J}_{17,20}=\overline{J}_{7,20}-q\overline{J}_{3,20}
\end{equation*}
and 
\begin{equation*}
J_{2,5}=\overline{J}_{9,20}-q^2\overline{J}_{19,20}=\overline{J}_{11,20}-q^3\overline{J}_{-1,20},
\end{equation*}
so
\begin{equation*}
J_{1,2}\overline{J}_{3,12}=J_{6,30}\Big ( -qJ_{2,5}\Big )+J_{12,30}J_{1,5}=J_{1,5}J_{12,30}-qJ_{2,5}J_{6,30}.
\end{equation*}
To prove (\ref{equation:id-2}), we substitute $x\rightarrow -q$ in (\ref{equation:DCH}) and use $\overline{J}_{1,1}=\overline{J}_{0,1}=2\overline{J}_{1,4}$:
\begin{align*}
2\overline{J}_{1,4}J_{3,6}&=\overline{J}_{1,1}J_{3,6}=J_{3,15}\Big ( qJ_{2,10}+qJ_{8,10}\Big )+J_{6,15}\Big ( J_{6,10}+J_{4,10}\Big )\\
&=2\Big ( J_{4,10}J_{6,15}+qJ_{2,10}J_{3,15}\Big ).\qedhere
\end{align*}
\end{proof}

\section{tenth-order mock theta functions and Appell--Lerch functions} \label{section:hecke}
We recall the definition for Hecke-type double-sums:
\begin{definition}  Let $x,y\in\mathbb{C}^*$ and $a,\ b,\ c$ be non-negative integers, then 
\begin{equation}
\Big ( \sum_{r,s\ge 0}-\sum_{r,s<0} \Big )(-1)^{r+s}x^ry^sq^{a\binom{r}{2}+brs+c\binom{s}{2}}.
\end{equation}
\end{definition}

Taking the $n=2$, $p=1$ specialization of \cite[Theorem $1.3$]{HM}, we have
\begin{proposition} \label{proposition:f232} For generic $x,y,z\in \mathbb{C}^*$
{\allowdisplaybreaks \begin{align}
f_{2,3,2}(x,y,q)&=j(x;q^2)m\Big (\frac{q^6y^2}{x^3},q^{10},-1\Big )
-yj(q^{3}x;q^2)m\Big (\frac{qy^2}{x^3},q^{10},-1\Big )\\
&\ \ \ \ \ +j(y;q^2)m\Big (\frac{q^6x^2}{y^3},q^{10},-1\Big )
-xj(q^{3}y;q^2)m\Big (\frac{qx^2}{y^3},q^{10},-1\Big )\notag\\
&\ \ \ \ \ -\frac{1}{\overline{J}_{0,10}}
\cdot  \frac{y}{qx}\cdot \frac{J_{5}^3j(-x^2/y^2;q^{2})j(q^{3}xy;q^{5})}
 {j(-q^{4}y^{3}/x^2;q^5)j(-q^{4}x^3/y^2;q^{5})}.\notag
\end{align}}%
\end{proposition}

Rewriting the respective Hecke-type double-sums from \cite{C1,C2}:
{\allowdisplaybreaks
\begin{align}
J_{1,2}\phi(q)&=f_{2,3,2}(q^2,q^2,q),\label{equation:phi-hecke}\\
J_{1,2}\psi(q)&=-q^2f_{2,3,2}(q^4,q^4,q),\label{equation:psi-hecke}\\
\overline{J}_{1,4}X(q)&=f_{2,3,2}(-q^3,-q^3,q^2),\label{equation:X-hecke}\\
\overline{J}_{1,4}(2-\chi(q))&=qf_{2,3,2}(-q^{-1},-q^{-1},q^2).\label{equation:chi-hecke}
\end{align}}%
\begin{corollary}\label{corollary:mxqz-forms} The following are true
{\allowdisplaybreaks \begin{align}
{\phi}(q)&=-q^{-1}m(q,q^{10},q)-q^{-1}m(q,q^{10},q^{2}),\label{equation:10th-phi(q)}\\
{\psi}(q)&=-m(q^3,q^{10},q)-m(q^3,q^{10},q^{3}),\label{equation:10th-psi(q)}\\
{X}(q)&=m(-q^2,q^{5},q)+m(-q^2,q^{5},q^{4}),\label{equation:10th-BigX(q)}\\
{\chi}(q)&=m(-q,q^{5},q^2)+m(-q,q^{5},q^{3}).\label{equation:10th-chi(q)}
\end{align}}%
\end{corollary}

We state a lemma:
\begin{lemma} \label{lemma:f232-X}We have
\begin{align}
D_2(-q^2,q^5,q,-1)&=q^{-2}\frac{J_{10}^3J_{5,10}\overline{J}_{12,20}}{\overline{J}_{2,5}\overline{J}_{0,20}J_{1,10}J_{4,10}},\\
D_2(-q^2,q^5,q^4,-1)&=q^{-2}\frac{J_{10}^{3}J_{5,10}J_{3,10}\overline{J}_{4,20}}{\overline{J}_{1,5}\overline{J}_{0,20}J_{1,10}^2J_{4,10}}.
\end{align}
\end{lemma}
\begin{proof} For the first identity, use Corollary \ref{corollary:msplitn2zprime}.  Note that one of the two theta quotients of 
(\ref{equation:msplit2}) vanishes. For the second identity, we use Corollary \ref{corollary:msplitn2zprime} to obtain
{\allowdisplaybreaks \begin{align*}
D_2&(-q^2,q^5,q^4,-1)\\
&= q^{-2}\frac{J_{10}^3J_{3,10}\overline{J}_{8,20}}{\overline{J}_{1,5}\overline{J}_{0,20}J_{1,10}J_{4,10}}
+q^{-1}\frac{J_{10}^3J_{2,10}\overline{J}_{2,20}}{\overline{J}_{1,5}\overline{J}_{0,20}J_{1,10}^2}\\
&=  q^{-2}\frac{J_{10}^3}{\overline{J}_{1,5}\overline{J}_{0,20}J_{1,10}^2J_{4,10}}
 \Big [J_{3,10}J_{1,10}\overline{J}_{8,20}+qJ_{2,10}J_{4,10}\overline{J}_{2,20} \Big]\\
 &=q^{-2}\frac{J_{10}^3}{\overline{J}_{1,5}\overline{J}_{0,20}J_{1,10}^2J_{4,10}} \frac{J_{20}}{J_{10}^2} 
 \Big [j(q^3;q^{10})j(q;q^{10})j(iq^{4};q^{10})j(-iq^{4};q^{10})\\
 &\ \ \ \ \ \ \ \ \ \ +qj(q^2;q^{10})j(q^4;q^{10})j(iq;q^{10})j(-iq;q^{10}) \Big]\\
&=q^{-2}\frac{J_{10}^3}{\overline{J}_{1,5}\overline{J}_{0,20}J_{1,10}^2J_{4,10}} \frac{J_{20}}{J_{10}^2}
 \Big [j(q^5;q^{10})j(q^3;q^{10})j(iq^2;q^{10})j(-iq^2;q^{10}) \Big],
\end{align*}}%
where in the last two equalities we have used (\ref{equation:j-mod-dec}) and then (\ref{equation:Weierstrass}) with $q\rightarrow q^{10}$ and $a=q^4$, $b=q^{2}$, $c=q$, $d=i$.  The result then follows from product rearrangements.
\end{proof}

\begin{proof}[Proof of Corollary \ref{corollary:mxqz-forms}] The proofs for (\ref{equation:10th-phi(q)}) and (\ref{equation:10th-psi(q)}) are similar, so we will only do the first identity.  Using Proposition \ref{proposition:f232} and Hecke sum identity (\ref{equation:phi-hecke}), we have 
{\allowdisplaybreaks \begin{align*}
f_{2,3,2}&(q^2,q^2,q)\\
&=-q^{-1}J_{1,2}m\Big (q,q^{10},-1\Big )
-q^{-1}J_{1,2}m\Big (q,q^{10},-1\Big )
 + \frac{q^{-1}J_{5}^3\overline{J}_{0,2}J_{2,5}}
 {\overline{J}_{0,10}\overline{J}_{1,5}^2}\\
&=-q^{-1}J_{1,2}m\Big (q,q^{10},q\Big )-q^{-1}J_{1,2}m\Big (q,q^{10},q^2\Big )&(\textup{by (\ref{equation:changing-z})})\\
&\ \ \ \ \ - \frac{q^{-1}J_{10}^3J_{1,2}\overline{J}_{2,10}}{\overline{J}_{0,10}J_{2,10}} 
 \Big [ \frac{1}{J_{1,10}}+ \frac{\overline{J}_{3,10}}{\overline{J}_{1,10}J_{3,10}}\Big ]
  + \frac{q^{-1}J_{5}^3\overline{J}_{0,2}J_{2,5}}
 {\overline{J}_{0,10}\overline{J}_{1,5}^2}\\
&=-q^{-1}J_{1,2}m\Big (q,q^{10},q\Big )-q^{-1}J_{1,2}m\Big (q,q^{10},q^2\Big )\\
&\ \ \ \ \ - \frac{q^{-1}J_{10}^3J_{1,2}\overline{J}_{2,10}}{\overline{J}_{0,10}J_{2,10}}
 \frac{\overline{J}_{1,10}J_{3,10}+ J_{1,10}\overline{J}_{3,10}}{J_{1,10}\overline{J}_{1,10}J_{3,10}}
   + \frac{q^{-1}J_{5}^3\overline{J}_{0,2}J_{2,5}}
 {\overline{J}_{0,10}\overline{J}_{1,5}^2}\\
&=-q^{-1}J_{1,2}m\Big (q,q^{10},q\Big )-q^{-1}J_{1,2}m\Big (q,q^{10},q^2\Big )&(\textup{by (\ref{equation:H1Thm1.2B})})\\
&\ \ \ \ \ - \frac{q^{-1}J_{10}^3J_{1,2}\overline{J}_{2,10}}{\overline{J}_{0,10}J_{2,10}}
  \frac{2J_{4,20}J_{12,20}}{J_{1,10}\overline{J}_{1,10}J_{3,10}}
  + \frac{q^{-1}J_{5}^3\overline{J}_{0,2}J_{2,5}}
{\overline{J}_{0,10}\overline{J}_{1,5}^2}\\ 
&=-q^{-1}j(q;q^2)m\Big (q,q^{10},q\Big )-q^{-1}j(q;q^2)m\Big (q,q^{10},q^2\Big ),
\end{align*}}%
where the last line follows by elementary product rearrangements.  The proofs for (\ref{equation:10th-BigX(q)}) and (\ref{equation:10th-chi(q)}) are similar, so we will only do the third identity.  Using Proposition \ref{proposition:f232}, the Hecke sum identity (\ref{equation:X-hecke}), and Lemma \ref{lemma:f232-X}, we have
{\allowdisplaybreaks \begin{align*}
&f_{2,3,2}(-q^3,-q^3,q^2)\\
&=\overline{J}_{1,4}m\Big (-q^{9},q^{20},-1\Big )
+q^{-3}\overline{J}_{1,4}m\Big (-q^{-1},q^{20},-1\Big )\\
&\ \ \ \ \ +\overline{J}_{1,4}m\Big (-q^{9},q^{20},-1\Big )
+q^{-3}\overline{J}_{1,4}m\Big (-q^{-1},q^{20},-1\Big )
 +  q^{-2} \frac{J_{10}^3\overline{J}_{0,4}J_{2,10}}
 {\overline{J}_{0,20}J_{1,10}^2}\\
&=\overline{J}_{1,4}m(-q^2,q^5,q)+\overline{J}_{1,4}m(-q^2,q^5,q^4) \\
&\ \ \ \ \ -q^{-2}\frac{J_{10}^3\overline{J}_{1,4}J_{5,10}}{\overline{J}_{0,20}J_{1,10}J_{4,10}}
\Big [ \frac{\overline{J}_{12,20}}{\overline{J}_{2,5}}+\frac{J_{3,10}\overline{J}_{4,20}}{\overline{J}_{1,5}J_{1,10}}\Big]
+  q^{-2}\frac{J_{10}^3\overline{J}_{0,4}J_{2,10}}{\overline{J}_{0,20}J_{1,10}^2}\\
&=\overline{J}_{1,4}m(-q^2,q^5,q)+\overline{J}_{1,4}m(-q^2,q^5,q^4) &(\textup{by (\ref{equation:j-mod-inc})})\\
&\ \ \ \ \ -q^{-2}\frac{J_{10}^3\overline{J}_{1,4}J_{5,10}}{\overline{J}_{0,20}J_{1,10}J_{4,10}}
\frac{J_{10}^2}{J_5} \Big [ \frac{\overline{J}_{12,20}}{\overline{J}_{2,10}\overline{J}_{3,10}}
+\frac{J_{3,10}\overline{J}_{4,20}}{\overline{J}_{1,10}\overline{J}_{6,10}J_{1,10}}\Big]
+  q^{-2}\frac{J_{10}^3\overline{J}_{0,4}J_{2,10}}{\overline{J}_{0,20}J_{1,10}^2}\\
&=\overline{J}_{1,4}m(-q^2,q^5,q)+\overline{J}_{1,4}m(-q^2,q^5,q^4) 
+  q^{-2}\frac{J_{10}^3\overline{J}_{0,4}J_{2,10}}{\overline{J}_{0,20}J_{1,10}^2}\\
&\ \ \ \ \ -q^{-2}\frac{J_{10}^3\overline{J}_{1,4}J_{5,10}}{\overline{J}_{0,20}J_{1,10}J_{4,10}}
\frac{J_{10}^2}{J_5} \Big [ \frac{\overline{J}_{12,20}\overline{J}_{1,10}\overline{J}_{6,10}J_{1,10}
+J_{3,10}\overline{J}_{2,10}\overline{J}_{3,10}\overline{J}_{4,20}}{\overline{J}_{2,10}\overline{J}_{3,10}\overline{J}_{1,10}\overline{J}_{6,10}J_{1,10}}\Big]\\
&=\overline{J}_{1,4}m(-q^2,q^5,q)+\overline{J}_{1,4}m(-q^2,q^5,q^4) 
+  q^{-2}\frac{J_{10}^3\overline{J}_{0,4}J_{2,10}}{\overline{J}_{0,20}J_{1,10}^2}&(\textup{by (\ref{equation:j-mod-dec})})\\
&\ \ \ \ \ -q^{-2}\frac{J_{10}^3\overline{J}_{1,4}J_{5,10}}{\overline{J}_{0,20}J_{1,10}J_{4,10}}
\frac{J_{10}^5}{J_5J_{20}^3}
\Big [ \frac{\overline{J}_{12,20}J_{2,20}\overline{J}_{6,20}\overline{J}_{16,20}
+\overline{J}_{2,20}\overline{J}_{12,20}J_{6,20}\overline{J}_{4,20}}{\overline{J}_{2,10}\overline{J}_{3,10}\overline{J}_{1,10}\overline{J}_{6,10}J_{1,10}}\Big]\\
&=\overline{J}_{1,4}m(-q^2,q^5,q)+\overline{J}_{1,4}m(-q^2,q^5,q^4) 
+  q^{-2}\frac{J_{10}^3\overline{J}_{0,4}J_{2,10}}{\overline{J}_{0,20}J_{1,10}^2}\\
&\ \ \ \ \ -q^{-2}\frac{J_{10}^3\overline{J}_{1,4}J_{5,10}}{\overline{J}_{0,20}J_{1,10}J_{4,10}}
\frac{J_{10}^5}{J_5J_{20}^3} 
 \frac{\overline{J}_{12,20}\overline{J}_{4,20}}{\overline{J}_{2,10}\overline{J}_{3,10}\overline{J}_{1,10}\overline{J}_{6,10}J_{1,10}} \Big [ J_{2,20}\overline{J}_{6,20}+\overline{J}_{2,20}J_{6,20}\Big]\\
&=\overline{J}_{1,4}m(-q^2,q^5,q)+\overline{J}_{1,4}m(-q^2,q^5,q^4) +  q^{-2}\frac{J_{10}^3\overline{J}_{0,4}J_{2,10}}{\overline{J}_{0,20}J_{1,10}^2}\\
&\ \ \ \ \ -q^{-2}\frac{J_{10}^3\overline{J}_{1,4}J_{5,10}}{\overline{J}_{0,20}J_{1,10}J_{4,10}}
\frac{J_{10}^5}{J_5J_{20}^3}
 \frac{\overline{J}_{12,20}\overline{J}_{4,20}}{\overline{J}_{2,10}\overline{J}_{3,10}\overline{J}_{1,10}\overline{J}_{6,10}J_{1,10}}\cdot 
 2J_{8,40}J_{24,40},&(\textup{by (\ref{equation:H1Thm1.2B})})
\end{align*}}%
and the result follows by elementary product rearrangements.
\end{proof}

\section{The six identities in terms of the $D_n(x,q,z,z')$ function}\label{section:Dn-forms}  We rewrite Ramanujan's six identities for the tenth-order mock theta functions.
\begin{lemma}  We have
\begin{align}
\psi(q)+q\phi(-q^4)+X(q^8)&=-D_2(q^3,q^{10},q^6,q^{-8})-D_2(q^3,q^{10},q^{4},q^{8}),\label{equation:id5-D2-expansion}\\
\phi(q)-q^{-1}\psi(-q^4)+q^{-2}\chi(q^8)&=-q^{-1}D_2(q,q^{10},q^8,q^{-24})-q^{-1}D_2(q,q^{10},q^2,q^{-16}).\label{equation:id6-D2-expansion}
\end{align}
\end{lemma}
\begin{proof}   The proofs for (\ref{equation:id5-D2-expansion}) and  (\ref{equation:id6-D2-expansion}) are similar, so we will only do the first.  Using (\ref{equation:10th-phi(q)}),  (\ref{equation:10th-psi(q)}),  and (\ref{equation:10th-BigX(q)}), we have
{\allowdisplaybreaks \begin{align*}
\psi(q)&+q\phi(-q^4)+X(q^8)\\
&\ =-m(q^3,q^{10},q)-m(q^3,q^{10},q^{3})+ q^{-3}m(-q^4,q^{40},-q^4)+q^{-3}m(-q^4,q^{40},q^{8})\\
&\ \ \ \ \ +m(-q^{16},q^{40},q^8)+m(-q^{16},q^{40},q^{32}),\\
&\ =-m(q^3,q^{10},q^6)-m(q^3,q^{10},q^{4})-q^{-7}m(-q^{-4},q^{40},q^{8})-q^{-7}m(-q^{-4},q^{40},q^{-8})\\
&\ \ \ \ \ +m(-q^{16},q^{40},q^8)+m(-q^{16},q^{40},q^{-8}),
\end{align*}}%
where we have used (\ref{equation:mxqz-fnq-newz}), (\ref{equation:mxqz-fnq-z}), (\ref{equation:mxqz-flip}).  The result then follows from (\ref{equation:D2-def}).
\end{proof}

\begin{lemma}  We have
\begin{align}
q^{2}\phi(q^9)&-\frac{\psi(\omega q)-\psi(\omega^2 q)}{\omega - \omega^2}\label{equation:id1-D3-expansion}\\
&\ =\frac{1}{\omega-\omega^2}\Big [ D_{3}(q^3,\omega q^{10},q^{3},q^{9})- D_{3}(q^3,\omega^2 q^{10},q^{3},q^{9})\notag\\
&\ \ \ \ \ \ \ \ \ \ \ \ \ \ \  +D_3(q^3,\omega q^{10},q^6,q^{18})-D_3(q^3,\omega^2 q^{10},q^6,q^{18}) \Big ],\notag\\
q^{-2}\psi(q^9)&+\frac{\omega \phi(\omega q)-\omega^2\phi(\omega^2 q)}{\omega - \omega^2}\label{equation:id2-D3-expansion}\\
&\  =-\frac{q^{-1}}{\omega-\omega^2}\Big [ D_3(\omega q, \omega q^{10}, q^{-3},q^{-9})
 - D_3(\omega^2 q, \omega^2 q^{10}, q^{-3},q^{-9})\notag\\
&\ \ \ \ \ \ \ \ \ \ \ \ \ \ \ +D_3(\omega q, \omega q^{10}, q^{-9},q^{-27})- D_3(\omega^2 q, \omega^2 q^{10}, q^{-9},q^{-27})\Big ].\notag
\end{align}
\end{lemma}
\begin{proof} 
Rewriting identity (\ref{equation:tenth-id-1}) with expansions (\ref{equation:10th-phi(q)}) and (\ref{equation:10th-psi(q)}) gives
{\allowdisplaybreaks \begin{align*}
&q^{2}\phi(q^9)-\frac{\psi(\omega q)-\psi(\omega^2 q)}{\omega - \omega^2}\\
&\ = -q^{-7}m(q^9,q^{90},q^9)-q^{-7}m(q^9,q^{90},q^{18})\\
&\ \ \ \ +\frac{1}{\omega-\omega^2}\Big [ m(q^{3},\omega q^{10},\omega q) +m(q^{3},\omega q^{10},q^{3}) -m(q^{3},\omega^2 q^{10},\omega^2q) -m(q^{3},\omega^2 q^{10},q^{3})\Big ]\\
&\ = -q^{-7}m(q^9,q^{90},q^9)-q^{-7}m(q^9,q^{90},q^{18})\\
&\ \ \ \ +\frac{1}{\omega-\omega^2}\Big [ m(q^{3},\omega q^{10}, q^6)
+m(q^{3},\omega q^{10},q^{3})-m(q^{3},\omega^2 q^{10},q^6)-m(q^{3},\omega^2 q^{10},q^{3})\Big ],
\end{align*}}%
where we have used (\ref{equation:mxqz-fnq-z}) and (\ref{equation:mxqz-fnq-newz}).  The result then follows from (\ref{equation:D3-def}).  The argument for (\ref{equation:id2-D3-expansion}) is similar but uses (\ref{equation:mxqz-flip}), (\ref{equation:mxqz-fnq-newz}), and (\ref{equation:mxqz-fnq-z}).
\end{proof}

\begin{lemma}  We have
 \begin{align}
X(q^9)&-\frac{\omega\chi(\omega q)-\omega^2 \chi(\omega^2 q)}{\omega-\omega^2 }\label{equation:id3-D3-expansion}\\
&\ =-\frac{1}{1-\omega}\Big [ D_3(-\omega q, \omega^2 q^5,-q^{-3},-q^{-9} )-\omega D_3(-\omega^2 q, \omega q^5,-q^{-3},-q^{-9} )\notag\\
&\ \ \ \ \ \ \ \ \ \ \ \ \ \ \ +D_3(-\omega q, \omega^2 q^5,q^{3},q^{9} )-\omega D_3(-\omega^2 q, \omega q^5,q^{3},q^{9} )\Big],\notag \\
\chi(q^9)&+q^2\frac{X(\omega q)- X(\omega^2 q)}{\omega-\omega^2 }\label{equation:id4-D3-expansion}\\
&\ =\frac{q^2}{\omega-\omega^2}\Big [
D_{3}(-\omega^2q^2,\omega ^2 q^5,q^6,q^{18})-D_{3}(-\omega q^2,\omega q^5,q^6,q^{18})\notag\\
&\ \ \ \ \ \ \ \ \ \ \ \ \ \ \ +D_{3}(-\omega^2 q^2,\omega^2 q^5,q^9,q^{27})-D_{3}(-\omega q^2,\omega q^5,q^9,q^{27})\Big ].\notag
\end{align}
\end{lemma}
\begin{proof} Rewriting identity (\ref{equation:tenth-id-3}) with expansions (\ref{equation:10th-BigX(q)}) and (\ref{equation:10th-chi(q)}) gives
{\allowdisplaybreaks \begin{align*}
X&(q^9)-\frac{\omega \chi(\omega q)-\omega^2 \chi(\omega^2 q)}{\omega-\omega^2 }\\
&=m(-q^{18},q^{45},q^9)+m(-q^{18},q^{45},q^{36})
 -\frac{1}{1-\omega}\Big [m(-\omega q,\omega^2q^5,\omega^2q^2)\\
&\ \ \ \ \ +m(-\omega q,\omega^2q^5,q^3)
 -\omega m(-\omega^2 q,\omega q^5,\omega q^2)-\omega m(-\omega^2 q,\omega q^5,q^3)\Big ]\\
&=m(-q^{18},q^{45},q^9)+m(-q^{18},q^{45},-q^{-9}) -\frac{1}{1-\omega}\Big [m(-\omega q,\omega^2q^5,-q^{-3})\\
&\ \ \ \ \ +m(-\omega q,\omega^2q^5,q^3)
 -\omega m(-\omega^2 q,\omega q^5,-q^{-3})-\omega m(-\omega^2 q,\omega q^5,q^3)\Big ],
\end{align*}}%
where we have used (\ref{equation:mxqz-fnq-z}) and (\ref{equation:mxqz-fnq-newz}).  The result then follows from (\ref{equation:D3-def}).  The proof of identity (\ref{equation:id4-D3-expansion}) is similar but uses (\ref{equation:mxqz-fnq-z}).
\end{proof}

\section{Specializations of the $D_n(x,q,z,z')$ function}\label{section:single-quotients}
We have the following technical lemmas:
\begin{lemma} \label{lemma:msplit} We have
\begin{align}
D_2(q^3,q^{10},q^6,q^{-8})&=-\frac{J_{20}^3\overline{J}_{14,20}J_{20,40}}{J_{1,10}J_{8,40}\overline{J}_{8,20}J_{6,20}},\\
D_2(q^3,q^{10},q^{4},q^{8})&= -q\cdot \frac{J_{20}^3\overline{J}_{18,20}J_{20,40}}{J_{7,10}J_{8,40}\overline{J}_{4,20}J_{6,20}}.
\end{align}
\end{lemma}
\begin{proof}  For each identity, use Corollary \ref{corollary:msplitn2zprime}.
\end{proof}

\begin{lemma} \label{lemma:msplit-fifth} We have
{\allowdisplaybreaks \begin{align}
D_2(q,q^{10},q^8,q^{-24})&=-q\cdot \frac{J_{20}^3\overline{J}_{6,20}J_{20,40}}{J_{9,10}J_{24,40}\overline{J}_{12,20}J_{18,20}},\\
D_2(q,q^{10},q^2,q^{-16})&=-q^2\cdot \frac{J_{20}^3\overline{J}_{2,20}J_{20,40}}{J_{3,10}J_{16,40}\overline{J}_{4,20}J_{2,20}}.
\end{align}}%
\end{lemma}
\begin{proof}  For each identity, use Corollary \ref{corollary:msplitn2zprime}.
\end{proof}

\begin{lemma} \label{lemma:tenth-id-1} We have
{\allowdisplaybreaks \begin{align}
D_{3}(q^3,q^{10},q^{3},q^{9})&=-q^{-3}\cdot \frac{J_{30}^7J_{12,30}}{J_{6,30}J_{9,30}J_{9,90}J_{18,30}}
\cdot \frac{1}
{J_{5,30}J_{7,30}J_{13,30}},\\
D_3(q^3,q^{10},q^6,q^{18})&=-q^{-3}\cdot 
\frac{J_{30}^7J_{12,30}}{J_{6,30}J_{9,30}J_{18,90}J_{27,30}}
\cdot 
\frac{1}{J_{4,30}J_{5,30}J_{14,30}}.
\end{align}}%
\end{lemma}
\begin{proof} For the first identity, we use Corollary \ref{corollary:msplitn3zprime} to  have
{\allowdisplaybreaks \begin{align*}
D_{3}(q^3,q^{10},q^{3},q^{9})&=\frac{q^9J_{30}^4}{J_{6,10}J_{9,90}J_{18,30}}
 \Big [q^{-8}\frac{J_{1,30}}{J_{13,30}}
-q^{-12}\frac{J_{11,30}}{J_{23,30}}\Big ]\\
&=\frac{J_{30}^3}{J_{10}}\frac{q^9J_{30}^4}{J_{6,30}J_{16,30}J_{26,30}J_{9,90}J_{18,30}} 
 \Big [q^{-8}\frac{J_{1,30}}{J_{13,30}}
-q^{-12}\frac{J_{11,30}}{J_{23,30}}\Big ]\\
&=-\frac{q^{-3}J_{30}^7}{J_{10,30}J_{6,30}J_{16,30}J_{26,30}J_{9,90}J_{18,30}}\cdot 
 \Big [\frac{J_{4,30}J_{10,30}J_{14,30}}
{J_{5,30}J_{7,30}J_{13,30}}\frac{J_{12,30}}{J_{9,30}} \Big ],
\end{align*}}%
where we have used (\ref{equation:j-mod-inc}) with $n=3$ followed by the relation (\ref{equation:Weierstrass}) with $q\rightarrow q^{30}$, $a= q^{16}$, $b= q^7$, $c=q^3$,  $d= q^2$.  The result follows from simplifying.  The second identity is similar but follows from $q\rightarrow q^{30}$, $a= q^{16}$, $b= q^{10}$, $c=q^9$,  $d= q^{4}$.
\end{proof}

\begin{lemma}  \label{lemma:tenth-id-2} We have
{\allowdisplaybreaks \begin{align*}
D_{3}(q,q^{10},q^{-9},q^{-27})&=-\frac{J_{30}^7}{J_{18,30}J_{27,90}J_{3,30}}
\cdot \frac{1}{J_{1,30}J_{5,30}J_{11,30}},\\
D_{3}(q,q^{10},q^{-3},q^{-9})&=-q^{-3}\cdot \frac{J_{30}^7}{J_{18,30}J_{9,90}J_{3,30}}
\cdot \frac{1}{J_{5,30}J_{7,30}J_{13,30}}.
\end{align*}}%
\end{lemma}
\begin{proof}    For the first identity, we use Corollary \ref{corollary:msplitn3zprime} to  have
{\allowdisplaybreaks \begin{align*}
D_{3}(q,q^{10},q^{-9},q^{-27})&=-\frac{J_{30}^4}{J_{2,10}J_{27,90}J_{24,30}}
\Big [ \frac{J_{23,30}}{J_{1,30}}
-q^{2}\frac{J_{13,30}}{J_{11,30}}\Big ]\\
&=-\frac{J_{30}^4}{J_{2,10}J_{27,90}J_{24,30}}
\Big [ \frac{J_{2,30}J_{6,30}J_{8,30}J_{10,30}}{J_{1,30}J_{3,30}J_{5,30}J_{11,30}}\Big ],
\end{align*}}%
where we have used the relation (\ref{equation:Weierstrass}) with $q\rightarrow q^{30}$, $a=q^9$, $b=q^4$, $c=q^2$,  $d=q$.  The result follows from simplification.  The second identity follows from $q\rightarrow q^{30}$, $a=q^9$, $b=q^4$, $c=q^2$,  $d=q$.
\end{proof}

\begin{lemma} \label{lemma:tenth-id-3} We have
{\allowdisplaybreaks \begin{align}
&D_{3}(-q,q^{5},-q^{-3},-q^{-9})=- \frac{J_{15}^7}{J_{12,15}\overline{J}_{9,45}\overline{J}_{3,15}}
\cdot \frac{1}{\overline{J}_{2,15}\overline{J}_{7,15}\overline{J}_{5,15}}, \\
&D_{3}(-q,q^5,q^3,q^9)=q^{-1}\cdot \frac{J_{15}^2J_{30}^4J_{3,15}}{J_{9,45}\overline{J}_{12,15}J_{12,30}}
\cdot  \frac{1}{J_{2,30}J_{8,30}J_{5,30}}.
\end{align}}%
\end{lemma}
\begin{proof} For the first identity, we use Corollary \ref{corollary:msplitn3zprime} to obtain
{\allowdisplaybreaks \begin{align*}
D_{3}(-q,q^{5},-q^{-3},-q^{-9})
&=-\frac{J_{15}^4}{J_{2,5}\overline{J}_{9,45}J_{6,15}}\Big [ 
 \frac{\overline{J}_{4,15}}{\overline{J}_{2,15}}
-q^{2}\frac{\overline{J}_{1,15}}{\overline{J}_{7,15}}\Big ]\\
&=-\frac{J_{15}^3}{J_5}\frac{J_{15}^4}{J_{2,15}J_{7,15}J_{12,15}\overline{J}_{9,45}J_{6,15}}\Big [ 
 \frac{\overline{J}_{4,15}}{\overline{J}_{2,15}}
-q^{2}\frac{\overline{J}_{1,15}}{\overline{J}_{7,15}}\Big ]\\
&=-\frac{J_{15}^7}{J_{5,15}J_{2,15}J_{7,15}J_{12,15}\overline{J}_{9,45}J_{6,15}}\Big [ 
\frac{J_{2,15}J_{5,15}J_{7,15}J_{6,15}}{\overline{J}_{2,15}\overline{J}_{7,15}\overline{J}_{5,15}\overline{J}_{3,15}}
\Big ],
\end{align*}}%
where we have used (\ref{equation:j-mod-inc}) with $n=3$ followed by (\ref{equation:Weierstrass}) with $q\rightarrow q^{15}$, $a=-q^7$, $b=q^5$, $c=q^3$, $d=-q^2$.   The proof for the second identity is similar but uses instead (\ref{equation:H1Thm1.0}).
\end{proof}

\begin{lemma} \label{lemma:tenth-id-4} We have
{\allowdisplaybreaks \begin{align*}
D_{3}(-q^2,q^5,q^6,q^{18})&=-q\cdot\frac{J_{30}^4J_{15}^2J_{6,15}}{J_{18,45}\overline{J}_{9,15}J_{24,30}}
\cdot \frac{1}{J_{4,30}J_{14,30}J_{5,30}},\\
D_{3}(-q^2,q^5,q^9,q^{27})&= q^2\cdot \frac{J_{30}J_{15}^5J_{3,15}}{J_{27,45}\overline{J}_{3,15}J_{12,30}}
\cdot \frac{1}{J_{1,15}J_{4,15}\overline{J}_{5,15}}.
\end{align*}}%
\end{lemma}
\begin{proof}  For both identities we use Corollary \ref{corollary:msplitn3zprime}.  For the first identity, we obtain
{\allowdisplaybreaks \begin{align*}
D_{3}(-q^2,q^5,q^6,q^{18})&=\frac{qJ_{15}^4\overline{J}_{5,15}}{\overline{J}_{2,5}J_{18,45}\overline{J}_{9,15}}
\Big [ q\frac{1}{J_{11,15}}
-\frac{1}{J_{1,15}}\Big ]\\
&=-\frac{qJ_{15}^4\overline{J}_{5,15}}{\overline{J}_{2,5}J_{18,45}\overline{J}_{9,15}}
 \frac{J_{6,15}J_{5,15}}{J_{2,15}J_{3,15}J_{7,15}},
\end{align*}}%
where we have used the relation with $q\rightarrow q^{15}$, $a=q^5$, $b=q^3$, $c=q^2$, $d=q$.  The second identity follows from $q\rightarrow q^{15}$, $a=-q^6$, $b=q^5$, $c=q^4$, $d=q$.
\end{proof}

\section{Proofs of identities (\ref{equation:RLN-id-five}) and (\ref{equation:RLN-id-six})}\label{section:five-six}

Using identity (\ref{equation:id5-D2-expansion}) and Lemma \ref{lemma:msplit}, we have
{\allowdisplaybreaks \begin{align*}
\psi(q)&+q\phi(-q^4)+X(q^8)\\
&\ \ \ \ \ =\frac{J_{20}^3}{J_{1,10}J_{8,40}}
\frac{\overline{J}_{14,20}J_{20,40}} {\overline{J}_{8,20}J_{6,20}}
+\frac{qJ_{20}^3}{J_{7,10}J_{8,40}}
 \frac{\overline{J}_{18,20}J_{20,40}}{\overline{J}_{4,20}J_{6,20}}\\
&\ \ \ \ \ =\frac{J_{20}^3J_{20,40}}{J_{8,40}J_{6,20}} \frac{1}{J_{1,10}J_{7,10}\overline{J}_{4,20}{\overline{J}_{8,20}}}
\Big [ 
\overline{J}_{14,20}J_{7,10}\overline{J}_{4,20}
+
q \overline{J}_{18,20}J_{1,10}\overline{J}_{8,20}\Big ]\\
&\ \ \ \ \ =\frac{J_{20}^5J_{20,40}}{J_{10}J_{8,40}J_{6,20}} \frac{1}{J_{1,10}J_{7,10}\overline{J}_{4,20}{\overline{J}_{8,20}}}
\Big [ \overline{J}_{4,10}J_{7,10} +q \overline{J}_{2,10}J_{1,10}\Big ]\\
&\ \ \ \ \ =\frac{J_{20}^5J_{20,40}}{J_{10}J_{8,40}J_{6,20}} \frac{1}{J_{1,10}J_{7,10}\overline{J}_{4,20}{\overline{J}_{8,20}}}
\Big [ j(-q;-q^5)j(-q^3;-q^5)\Big ],
\end{align*}}%
where we have used (\ref{equation:j-mod-inc}) for the penultimate equality and (\ref{equation:H1Thm1.1}) for the last equality.  The result then follows from product rearrangements.

Using (\ref{equation:id6-D2-expansion}) and Lemma \ref{lemma:msplit-fifth} gives
{\allowdisplaybreaks \begin{align*}
\phi(q)&-q^{-1}\psi(-q^4)+q^{-2}\chi(q^8)\\
&=\frac{J_{20}^3J_{20,40}}{J_{24,40}J_{2,20}}\cdot \frac{1}{J_{3,10}\overline{J}_{4,40}J_{9,10}\overline{J}_{12,20}}\cdot
\Big [\overline{J}_{6,20}J_{3,10}\overline{J}_{4,20} +q\overline{J}_{2,20}J_{9,10}\overline{J}_{12,20}\Big ]\\
&=\frac{J_{20}^5J_{20,40}}{J_{10}J_{24,20}J_{2,20}}\cdot \frac{1}{J_{3,10}\overline{J}_{4,20}J_{9,10}\overline{J}_{12,20}}\cdot
\Big [J_{3,10}\overline{J}_{4,10} +q\overline{J}_{2,10}J_{9,10}\Big ]\\
&=\frac{J_{20}^5J_{20,40}}{J_{10}J_{24,20}J_{2,20}}\cdot \frac{1}{J_{3,10}\overline{J}_{4,20}J_{9,10}\overline{J}_{12,20}}\cdot
\Big [j(-q;-q^5)j(q^2;-q^5)\Big ],
\end{align*}}%
where we have used (\ref{equation:j-mod-inc}) for the penultimate equality and (\ref{equation:H1Thm1.1}) for the last equality.  The result then follows from product rearrangements.

\section{Proofs of identities (\ref{equation:tenth-id-1}) and (\ref{equation:tenth-id-2})}\label{section:first-and-second}
To prove identity (\ref{equation:tenth-id-1}), we use identity (\ref{equation:id1-D3-expansion}) and Lemma \ref{lemma:tenth-id-1} to obtain
{\allowdisplaybreaks \begin{align*}
q^{2}&\phi(q^9)-\frac{\psi(\omega q)-\psi(\omega^2 q)}{\omega - \omega^2}\\
&=  - \frac{1}{\omega-\omega^2} \frac{q^{-3}J_{30}^7J_{12,30}}{J_{6,30}J_{9,30}J_{9,90}J_{18,30}}
 \Big [ 
\frac{1}{j( \omega^2q^5 ;q^{30})j(\omega q^{7};q^{30})j(\omega q^{13};q^{30})} \\
& \ \ \ \ \ \ \ \ \ \ - \frac{1}{j( \omega q^5 ;q^{30})j(\omega^2 q^{7};q^{30})j(\omega^2 q^{13};q^{30})}\Big ]\\
&\ \ \ \ \ - \frac{1}{\omega-\omega^2}\frac{q^{-3}J_{30}^7J_{12,30}}{J_{6,30}J_{9,30}J_{18,90}J_{27,30}}
\Big [ 
 \frac{1}{j( \omega q^4;q^{30})j( \omega^2q^{5};q^{30})j(  \omega^2q^{14};q^{30})} \\
&\ \ \ \ \ \ \ \ \ \ - \frac{1}{j( \omega^2 q^4;q^{30})j( \omega q^{5};q^{30})j(  \omega q^{14};q^{30})}\Big ]\\
&= \frac{1}{\omega-\omega^2} \frac{q^{-3}J_{30}^7J_{12,30}}{J_{6,30}J_{9,30}J_{9,90}J_{18,30}}
  \frac{J_{90}^3}{J_{30}^9}\frac{J_{5,30}J_{7,30}J_{13,30}}{J_{15,90}J_{21,90}J_{39,90}}\cdot \\
&\ \ \ \ \ \ \ \ \ \ \cdot \Big [ j( \omega^2q^5 ;q^{30})j(\omega q^{7};q^{30})j(\omega q^{13};q^{30})
- j( \omega q^5 ;q^{30})j(\omega^2 q^{7};q^{30})j(\omega^2 q^{13};q^{30})\Big ] \\
&\ \ \ \ \ \  + \frac{1}{\omega-\omega^2}\frac{q^{-3}J_{30}^7J_{12,30}}{J_{6,30}J_{9,30}J_{18,90}J_{27,30}}
 \frac{J_{90}^3}{J_{30}^9}\frac{J_{4,30}J_{5,30}J_{14,30}}{J_{12,90}J_{15,90}J_{52,90}}\cdot  \\
&\ \ \ \ \ \ \ \ \ \  \cdot \Big [ j( \omega q^4;q^{30})j( \omega^2q^{5};q^{30})j(  \omega^2q^{14};q^{30}) 
 - j( \omega^2 q^4;q^{30})j( \omega q^{5};q^{30})j(  \omega q^{14};q^{30})\Big ],
\end{align*}}%
where we have pulled fractions over a common denominator.  Using the relation (\ref{equation:Weierstrass}) with $q\rightarrow q^{30}$, $a=q^{12}$, $b=q^{10}$, $c=\omega ^2 q^5 $, $d=\omega q^5$, and also $q\rightarrow q^{30}$, $a=q^{9}$, $b=q^{10}$, $c=\omega ^2 q^5 $, $d=\omega q^5 $, we have
{\allowdisplaybreaks \begin{align*}
q^2&\phi(q^9)-\frac{\psi(\omega q)- \psi(\omega^2 q)}{\omega -\omega^2}\\
&= \frac{1}{\omega-\omega^2} \frac{q^{-3}J_{30}^7J_{12,30}}{J_{6,30}J_{9,30}J_{9,90}J_{18,30}}
  \frac{J_{90}^3}{J_{30}^9}\frac{J_{5,30}J_{7,30}J_{13,30}}{J_{15,90}J_{21,90}J_{39,90}}
   \Big [ \omega q^5  \frac{ J_{22,30}J_{2,30}J_{10,30}j(\omega;q^{30})}{j(q^{15}\omega;q^{30})}\Big ] \\
&\ \ \ \ \ \  + \frac{1}{\omega-\omega^2}\frac{q^{-3}J_{30}^7J_{12,30}}{J_{6,30}J_{9,30}J_{18,90}J_{27,30}}
 \frac{J_{90}^3}{J_{30}^9}\frac{J_{4,30}J_{5,30}J_{14,30}}{J_{12,90}J_{15,90}J_{52,90}}\cdot  \\
&\ \ \ \ \ \ \ \ \ \  \cdot \Big [\omega q^5 \frac{J_{19,30}j(q^{-1};q^{30})J_{10,30}j(\omega;q^{30})}{j(q^{15}\omega;q^{15})}\Big ]\\
&= q^2\frac{J_{30}^2}{J_{9,30}}
  \frac{J_{2,5}J_{15}}{J_{6,15}J_{3,15}}
  -q \frac{J_{30}^2}{J_{9,30}} \frac{J_{18,30}}{J_{6,30}}
   \frac{J_{1,5}}{J_{6,15}} \frac{J_{15}}{J_{3,15}},
\end{align*}}%
where the last line follows from elementary simplification.  Proving identity (\ref{equation:tenth-id-1}) thus reduces to showing
\begin{equation}
q^2\frac{J_{30}^2}{J_{9,30}}\frac{J_{2,5}J_{15}}{J_{6,15}J_{3,15}}
   -q \frac{J_{30}^2}{J_{9,30}} \frac{J_{18,30}}{J_{6,30}} \frac{J_{1,5}}{J_{6,15}} \frac{J_{15}}{J_{3,15}}
   =-q\frac{J_{1,2}}{J_{3,6}}\frac{J_{3,15}J_{6}}{J_{3}},\label{equation:tenth-id-1-almost}
\end{equation}
which is obtained by dividing identity (\ref{equation:id-1}) by $J_{3,15}J_{6,15}^2/J_{15}^2$.

To prove identity (\ref{equation:tenth-id-2}), we use identity (\ref{equation:id2-D3-expansion}) and Lemma \ref{lemma:tenth-id-2} to obtain
{\allowdisplaybreaks \begin{align*}
&q^{-2}\psi(q^9)+\frac{\omega \phi(\omega q)-\omega^2\phi(\omega^2 )}{\omega - \omega^2}\\
&\ =\frac{q^{-1}}{\omega-\omega^2} \frac{J_{30}^7}{J_{18,30}J_{27,90}J_{3,30}} 
\Big [ \frac{1}{j(\omega q;q^{30})j(\omega ^2q^5;q^{30})j(\omega^2 q^{11};q^{30})}\\
&\ \ \ \ \ \ \ \ \ \  - \frac{1}{j(\omega^2q;q^{30})j(\omega q^5;q^{30})j(\omega q^{11};q^{30})}\Big ] \\
&\ \ \ \ \  +\frac{q^{-4}}{\omega-\omega^2} \frac{J_{30}^7}{J_{18,30}J_{9,90}J_{3,30}}
\Big [ \frac{1}{j(\omega^2 q^5;q^{30})j(\omega q^7;q^{30})j(\omega q^{13};q^{30})}\\
&\ \ \ \ \ \ \ \ \ \ -  \frac{1}{j(\omega q^5;q^{30})j(\omega^2q^7;q^{30})j(\omega^2 q^{13};q^{30})}\Big ]\\
&\ =\frac{q^{-1}}{\omega-\omega^2} \frac{J_{30}^7}{J_{18,30}J_{27,90}J_{3,30}}
\frac{J_{90}^3J_{1,30}J_{5,30}J_{11,30}}{J_{30}^9J_{3,90}J_{15,90}J_{33,90}} \cdot\\
&\ \ \ \ \ \ \ \ \ \ \cdot \Big [ j(\omega^2q;q^{30})j(\omega q^5;q^{30})j(\omega q^{11};q^{30})
-j(\omega q;q^{30})j(\omega ^2q^5;q^{30})j(\omega^2 q^{11};q^{30})\Big ] \\
&\ \ \ \ \  +\frac{q^{-4}}{\omega-\omega^2} \frac{J_{30}^7}{J_{18,30}J_{9,90}J_{3,30}}\frac{J_{90}^3J_{5,30}J_{7,30}J_{13,30}}{J_{30}^9J_{15,90}J_{21,90}J_{39,90}}\cdot \\
&\ \ \ \ \ \ \ \ \ \ \ \ \ \ \ \cdot \Big [ j(\omega q^5;q^{30})j(\omega^2q^7;q^{30})j(\omega^2 q^{13};q^{30})
-j(\omega^2 q^5;q^{30})j(\omega q^7;q^{30})j(\omega q^{13};q^{30}) \Big ].
\end{align*}}%
Using the relation (\ref{equation:Weierstrass}) with $q\rightarrow q^{30}$, $a=q^{10}$, $b=q^6$, $c=\omega^2 q^5$, $d=\omega q^5$ and also $q\rightarrow q^{30}$, $a=q^{10}$, $b=q^{12}$, $c=\omega^2q^5$, $d=\omega q^5$, yields
{\allowdisplaybreaks \begin{align*}
&q^{-2}\psi(q^9)+\frac{\omega \phi(\omega q)-\omega^2\phi(\omega^2 )}{\omega - \omega^2}\\
&\ =\frac{q^{-1}}{\omega-\omega^2} \frac{J_{30}^7}{J_{18,30}J_{27,90}J_{3,30}}
\frac{J_{90}^3J_{1,30}J_{5,30}J_{11,30}}{J_{30}^9J_{3,90}J_{15,90}J_{33,90}} 
 \Big [\frac{\omega q J_{16,30}J_{4,30}j(\omega;q^{30})J_{10,30}}{j(\omega q^{15},q^{30})}\Big ] \\
&\ \ \ \ \  +\frac{q^{-4}}{\omega-\omega^2} \frac{J_{30}^7}{J_{18,30}J_{9,90}J_{3,30}}\frac{J_{90}^3J_{5,30}J_{7,30}J_{13,30}}{J_{30}^9J_{15,90}J_{21,90}J_{39,90}}
 \Big [ -\frac{\omega q^5J_{22,20}J_{2,20}j(\omega;q^{30})J_{10,30}}{j(\omega q^{15},q^{30})} \Big ]\\
 &\ =\frac{J_{1,5}J_{15}}{J_{3,15}J_{6,15}}\frac{J_{30}^2}{J_{3,30}}-q\frac{J_{2,5}J_{15}}{J_{3,15}^2}\frac{J_{30}^2}{J_{9,30}},
\end{align*}}%
where the last line follows from simplification.  Thus proving (\ref{equation:tenth-id-2}) is equivalent to showing
\begin{equation}
\frac{J_{1,5}J_{15}}{J_{3,15}J_{6,15}}\frac{J_{30}^2}{J_{3,30}}-q\frac{J_{2,5}J_{15}}{J_{3,15}^2}\frac{J_{30}^2}{J_{9,30}}
=\frac{J_{1,2}}{J_{3,6}}\frac{J_{6,15}J_{6}}{J_{3}}
\end{equation}
which is obtained by dividing identity (\ref{equation:id-1}) by $J_{3,15}^2J_{6,15}/J_{15}^2$.

\section{Proofs of identities (\ref{equation:tenth-id-3}) and (\ref{equation:tenth-id-4})}  \label{section:third-and-fourth}

To prove identity (\ref{equation:tenth-id-3}), we use identity (\ref{equation:id3-D3-expansion}) and Lemma \ref{lemma:tenth-id-3} to obtain
{\allowdisplaybreaks \begin{align*}
X&(q^9)-\frac{\omega\chi(\omega q)-\omega^2 \chi(\omega^2 q)}{\omega-\omega^2 }\\
&=\frac{1}{1-\omega}  \frac{J_{15}^7}{J_{12,15}\overline{J}_{9,45}\overline{J}_{3,15}}
 \Big [  \frac{1}{j(-\omega^2 q^2;q^{15})j(-\omega q^7;q^{15})j(-\omega^2 q^5;q^{15})}\\
&\ \ \ \ \  -  \frac{\omega }{j(-\omega q^2;q^{15})j(-\omega ^2q^7;q^{15})j(-\omega q^5;q^{15})}\Big ]  \\
&\ \  -\frac{\omega^2}{1-\omega} \frac{1}{q} \frac{J_{15}^2J_{30}^4J_{3,15}}{J_{9,45}\overline{J}_{12,15}J_{12,30}} \Big[  \frac{1 }{j(\omega^2 q^2;q^{30})j(\omega^2 q^{8};q^{30})j(\omega ^2 q^5;q^{30})}\\
&\ \ \ \ \  -\frac{1}{j(\omega q^2;q^{30})j(\omega q^{8};q^{30})j(\omega q^5;q^{30})} \Big] \\
&=\frac{1}{1-\omega}  \frac{J_{15}^7}{J_{12,15}\overline{J}_{9,45}\overline{J}_{3,15}}
 \frac{J_{45}^3\overline{J}_{2,15}\overline{J}_{5,15}\overline{J}_{7,15}}{J_{15}^9\overline{J}_{6,45}\overline{J}_{15,45}\overline{J}_{21,45}}
 \Big [ j(-\omega q^2;q^{15})j(-\omega ^2q^7;q^{15})j(-\omega q^5;q^{15})\\
&\ \ \ \ \ \ \ \ \ \   -\omega j(-\omega^2 q^2;q^{15})j(-\omega q^7;q^{15})j(-\omega^2 q^5;q^{15}) \Big ]  \\
&\ \  -\frac{\omega^2}{1-\omega} \frac{1}{q} \frac{J_{15}^2J_{30}^4J_{3,15}}{J_{9,45}\overline{J}_{12,15}J_{12,30}}\frac{J_{90}^3}{J_{30}^9}\frac{J_{2,30}J_{8,30}J_{5,30}}{J_{6,90}J_{24,90}J_{15,90}}\cdot \\
&\ \ \ \ \  \cdot \Big[j(\omega q^2;q^{30})j(\omega q^{8};q^{30})j(\omega q^5;q^{30})
- j(\omega^2 q^2;q^{30})j(\omega^2 q^{8};q^{30})j(\omega ^2 q^5;q^{30}) \Big].
\end{align*}}%
Using the relation (\ref{equation:Weierstrass}) with $q\rightarrow q^{15}$, $a=q^{10}$, $b=-\omega q^5 $, $c=-\omega^2 q^5$, $d=q^3$, and with $q\rightarrow q^{30}$, $a=\omega q^5 $, $b=\omega^2 q^5$, $c=q^3$, $d=\omega $, yields
{\allowdisplaybreaks \begin{align*}
X&(q^9)-\frac{\omega \chi(\omega q)-\omega^2 \chi(\omega^2 q)}{\omega-\omega^2 }\\
&=\frac{1}{1-\omega}  \frac{J_{15}^7}{J_{12,15}\overline{J}_{9,45}\overline{J}_{3,15}} 
 \frac{J_{45}^3\overline{J}_{2,15}\overline{J}_{5,15}\overline{J}_{7,15}}{J_{15}^9\overline{J}_{6,45}\overline{J}_{15,45}\overline{J}_{21,45}} \Big [ \frac{J_{13,15}J_{7,15}J_{10,15}j(\omega^2;q^{15})}{j(-\omega^2 q^{15};q^{15})} \Big ]  \\
&\ \ \ \ \ -\frac{\omega^2}{1-\omega} \frac{1}{q} \frac{J_{15}^2J_{30}^4J_{3,15}}{J_{9,45}\overline{J}_{12,15}J_{12,30}} \frac{J_{90}^3}{J_{30}^9}\frac{J_{2,30}J_{8,30}J_{5,30}}{J_{6,90}J_{24,90}J_{15,90}}\cdot \\
&\ \ \ \ \ \ \ \ \ \ \cdot \Big[\frac{\omega^2 q^2 J_{10,30}j(\omega^2;q^{30})j(\omega q^3;q^{30})j(\omega^2 q^3;q^{30})}{J_{5,30}}\Big]\\
&=  \frac{J_{4,30}J_{14,30}}{\overline{J}_{6,15}}
  \frac{J_{10}J_{15}^2}{J_{6,30}J_{30}^2}
    + q \frac{J_{2,30}J_{8,30}}{\overline{J}_{3,15}}\frac{J_{10}J_{15}^2}{J_{6,30}J_{30}^2},
\end{align*}}%
where the last line follows from simplification.  Thus proving (\ref{equation:tenth-id-3}) is equivalent to showing
\begin{equation}
\frac{J_{4,30}J_{14,30}}{\overline{J}_{6,15}}
  \frac{J_{10}J_{15}^2}{J_{6,30}J_{30}^2}
    + q \frac{J_{2,30}J_{8,30}}{\overline{J}_{3,15}}\frac{J_{10}J_{15}^2}{J_{6,30}J_{30}^2}
    =\frac{\overline{J}_{1,4}}{\overline{J}_{3,12}}\frac{J_{18,30}J_{3}}{J_{6}},
\end{equation}
which is obtained by dividing identity (\ref{equation:id-2}) by $J_{6,30}^2J_{12,30}/J_{30}^2$.

To prove identity (\ref{equation:tenth-id-4}), we use identity (\ref{equation:id4-D3-expansion}) and Lemma \ref{lemma:tenth-id-4} to obtain
{\allowdisplaybreaks \begin{align*}
\chi&(q^9)+q^2\frac{X(\omega q)- X(\omega^2 q)}{\omega-\omega^2 }\\
&=-\frac{q^3}{\omega-\omega^2}\frac{J_{30}^4J_{15}^2J_{6,15}}{J_{18,45}\overline{J}_{9,15}J_{24,30}}\cdot \\
&\ \ \ \ \  \cdot\Big [  \frac{\omega }{j(\omega q^4;q^{30})j(\omega^2  q^{14};q^{30})j(\omega^2 q^5;q^{30})}
-  \frac{\omega^2}{j(\omega^2 q^4;q^{30})j(\omega  q^{14};q^{30})j(\omega q^5;q^{30})}\Big ]\\
&\ \ +\frac{q^4}{\omega-\omega^2}\frac{J_{30}J_{15}^5J_{3,15}}{J_{27,45}\overline{J}_{3,15}J_{12,30}}\cdot\\
&\ \ \ \ \  \cdot \Big [\frac{\omega^2}{j(\omega q;q^{15})j(\omega  q^{4};q^{15})j(-\omega^2 q^5;q^{15})} 
-\frac{\omega }{j(\omega^2 q;q^{15})j(\omega^2  q^{4};q^{15})j(-\omega q^5;q^{15})} \Big ]\\
&=-\frac{q^3}{1-\omega}\frac{J_{30}^4J_{15}^2J_{6,15}}{J_{18,45}\overline{J}_{9,15}J_{24,30}}
\frac{J_{90}^3J_{4,30}J_{14,30}J_{5,30}}{J_{30}^9J_{12,90}J_{42,90}J_{15,90}} \cdot \\
&\ \ \ \ \  \cdot\Big [ j(\omega^2 q^4;q^{30})j(\omega  q^{14};q^{30})j(\omega q^5;q^{30})
-\omega j(\omega q^4;q^{30})j(\omega ^2 q^{14};q^{30})j(\omega^2 q^5;q^{30})\Big ]\\
&\ \ -\frac{q^4}{1-\omega}\frac{J_{30}J_{15}^5J_{3,15}}{J_{27,45}\overline{J}_{3,15}J_{12,30}}
 \frac{J_{45}^3J_{1,15}J_{4,15}\overline{J}_{5,15}}{J_{15}^9J_{3,45}J_{12,45}\overline{J}_{15,45}}\cdot \\
&\ \ \ \ \  \cdot \Big [j(\omega q;q^{15})j(\omega  q^{4};q^{15})j(-\omega^2 q^5;q^{15})
-\omega j(\omega^2 q;q^{15})j(\omega^2  q^{4};q^{15})j(-\omega q^5;q^{15}) \Big ].
\end{align*}}%
Using the relation (\ref{equation:Weierstrass}) with $q\rightarrow q^{30}$, $a=q^{9}$, $b=\omega^2 q^5 $, $c=\omega q^5$, $d=\omega $, and with $q\rightarrow q^{15}$, $a=q^9 $, $b=\omega q^5$, $c=\omega^2 q^5$, $d=-q^5 $, yields
{\allowdisplaybreaks \begin{align*}
\chi&(q^9)+q^2\frac{X(\omega q)- X(\omega^2 q)}{\omega-\omega^2 }\\
&=-\frac{q^3}{1-\omega}\frac{J_{30}^4J_{15}^2J_{6,15}}{J_{18,45}\overline{J}_{9,15}J_{24,30}}
\frac{J_{90}^3J_{4,30}J_{14,30}J_{5,30}}{J_{30}^9J_{12,90}J_{42,90}J_{15,90}}
\Big [ \frac{j(\omega q^9;q^{30})j(\omega^2q^9;q^{30})j(\omega;q^{30})J_{10,30}}{J_{5,30}}\Big ]\\
&\ \ \ \ \ \ \ \ \ \ -\frac{q^4}{1-\omega}\frac{J_{30}J_{15}^5J_{3,15}}{J_{27,45}\overline{J}_{3,15}J_{12,30}}
 \frac{J_{45}^3J_{1,15}J_{4,15}\overline{J}_{5,15}}{J_{15}^9J_{3,45}J_{12,45}\overline{J}_{15,45}}
  \Big [\frac{\overline{J}_{1,15}\overline{J}_{4,15}J_{5,15}j(\omega^2;q^{15})}{j(-\omega;q^{15})}\Big ]\\
&=-q^3\cdot \frac{J_{4,10}}{\overline{J}_{6,15}}\frac{J_{15}^2J_{30}}{J_{12,30}J_{6,30}}
-q^4\cdot \frac{J_{2,10}}{\overline{J}_{3,15}}\frac{J_{15}^2J_{30}}{J_{12,30}^2},
\end{align*}}%
where the last line follows from simplification.  Thus proving (\ref{equation:tenth-id-4}) is equivalent to showing
{\allowdisplaybreaks \begin{align}
-q^3 \frac{J_{4,10}}{\overline{J}_{6,15}}\frac{J_{15}^2J_{30}}{J_{12,30}J_{6,30}}
-q^4 \frac{J_{2,10}}{\overline{J}_{3,15}}\frac{J_{15}^2J_{30}}{J_{12,30}^2}
=-q^3\frac{\overline{J}_{1,4}}{\overline{J}_{3,12}}\frac{J_{6,30}J_{3}}{J_{6}},
\end{align}}%
which is obtained by dividing identity (\ref{equation:id-2}) by $J_{6,30}J_{12,30}^2/J_{30}^2$.

\section*{Acknowledgements}
We would like to thank George Andrews for suggesting the problem to find short proofs for identities (\ref{equation:RLN-id-five}) and (\ref{equation:RLN-id-six}), this in turn led to the present paper.   We would also like to thank Dean Hickerson for his help in finding new proofs of identities (\ref{equation:id-1}) and (\ref{equation:id-2}).   The new proofs replace the old proofs that relied on modularity.

\end{document}